\documentclass[a4paper, 10pt, conference]{IEEEtran}
\IEEEoverridecommandlockouts

\usepackage[utf8]{inputenc} 
\usepackage[T1]{fontenc}    
\usepackage[centertags]{amsmath}
\usepackage{amssymb,amsthm}
\usepackage[hidelinks]{hyperref}       
\usepackage{url}            
\usepackage{booktabs}       
\usepackage{multirow}
\usepackage{amsfonts}       
\usepackage{graphicx,subfigure}
\usepackage{array} 
\usepackage{nicefrac}       
\usepackage{microtype}      
\usepackage{xcolor}         
\usepackage{algorithm}
\usepackage{algorithmic}
\usepackage{cleveref}
\graphicspath{{./Figs/}}

\DeclareMathOperator*{\argmin}{arg\,min}
\newcommand{\norm}[1]{\ensuremath{\left\| #1 \right\|}}

\newcommand{\braces}[1]{\ensuremath{\left\{ #1 \right\}}}

\renewcommand{\Re}{\ensuremath{\mathbb{R}}}

\newtheorem{theorem}{Theorem}

\newtheorem{remark}{Remark}

\newtheorem{corollary}{Corollary}
\newtheorem{assumption}{Assumption}

\title{\LARGE \bf
	\vspace{0.2in}
	Iterative Supervised Learning for Regression with Constraints
	\vspace{-0.1in}
}

\author{Tejaswi K. C. and Taeyoung Lee \vspace{-0.1in}
	\thanks{Mechanical and Aerospace Engineering, The George Washington University, Washington DC 20052 {\tt kctejaswi999@gmail.com,tylee@gwu.edu}}%
}

\begin{document}

\maketitle

\begin{abstract}
Regression in supervised learning often requires the enforcement of constraints to ensure that the trained models are consistent with the underlying structures of the input and output data.
This paper presents an iterative procedure to perform regression under arbitrary constraints. 
It is achieved by alternating between a learning step and a constraint enforcement step, to which an affine extension function is incorporated.
We show this leads to a contraction mapping under mild assumptions, from which the convergence is guaranteed analytically. 
The presented proof of convergence in regression with constraints is the unique contribution of this paper.
Furthermore, numerical experiments illustrate improvements in the trained model in terms of the quality of regression, the satisfaction of constraints, and also the stability in training, when compared to other existing algorithms. 
\end{abstract}

\section{Introduction}

Enforcing constraints on supervised learning is critical when the underlying structures of the data should be respected in the trained model, or when it is required to overcome a bias in the data set.
For instance, \cite{yang2020learning} has studied constraints caused by length, angle, or collision with projection when predicting the motion of a physical system with neural networks.
In~\cite{berk2017convex}, fairness with respect to protected features, such as race or gender, is addressed in socially sensitive decision making. 
Further, it has been illustrated by~\cite{borghesi2020improving} that the performance of deep learning can be improved by integrating the domain knowledge in the form of constraints.
As such, imposing constraints is desirable in injecting our prior knowledge of the model, which is encoded indirectly in the data, to supervised learning explicitly.  

One of the common techniques to implement constraints is augmenting the loss function with an additional penalty on the violation of the constraints, as presented by~\cite{mehta2018towards} and \cite{diligenti2017integrating}.
On the other hand, constraints have also been implemented directly as hard constraints that should be satisfied strictly. 
Imposing hard constraints on deep neural network is presented by~\cite{marquez2017imposing} after customizing large-scale optimization techniques.
Alternatively, \cite{nandwani2019primal} handles output label restrictions through a Lagrangian based formulation.
Both of these approaches based on additional regularization terms or hard constrained optimization involve the process of actively adjusting model parameters in training. 
In other words, the possibly conflicting goals of regression and constraint enforcement should be addressed simultaneously. 
This may hinder the efficiency of the training procedure, while making it susceptible to various numerical issues. 

Recently, an iterative procedure has been proposed by \cite{detassis2020teaching}, where the constraints are enforced by adjusting the target, instead of manipulating the model parameters directly, thereby addressing the aforementioned challenges. 
The desirable feature is that any supervised learning technique that is developed without constraint consideration can be adopted, in conjunction with nonlinear constrained optimization tools. 
However, this approach is heuristic in the sense that there is no analytical assurance for convergence through iterations, while its performance is illustrated with several numerical examples. 
In fact, it is challenging to present a convergence property in any supervised learning with constraints. 

The main objective of this paper is to establish a certain convergence guarantee in regression with constraints. 
We follow the procedure presented by~\cite{detassis2020teaching}, where the target is adjusted to satisfy constraints. 
More specifically, the ideal target is projected to the intersection between the set of possible outputs from the chosen model and the set of feasible outputs. 
Then the model parameters are optimized to the adjusted target, and these two steps are repeated. 
The proposed approach is motivated by the alternating projections~\cite{cheney1959proximity,boyd2003alternating} and Dykstra's algorithm~\cite{bauschke1994dykstra}. 
In particular, the two steps of iterations in adjusting the target, and in training the model are considered as certain projection operators, 
from which convergence is established by the Banach fixed point theorem~\cite{ciesielski2007stefan}.  

The desirable feature is that we have a certain assurance of convergence in regression with constraints. 
Another interesting feature is its general formulation: as discussed above, this framework can be integrated with any supervised learning technique.
And, it further addresses the challenges of adjusting the model parameters to the satisfaction of the constraints, while performing regression simultaneously.
One downside is that we cannot enforce the constraint strictly as hard constraints, but there is a design parameter that provides a trade-off between regression and constraint satisfaction.

Numerical experiments demonstrate that the proposed approach improves the regression performance in the similar level of constraint violation. 
More importantly, it exhibits more consistent results over five-fold validations.
As such, the proposed convergence proof is actually beneficial in numerical implementations. 

This paper is organized as follows.
The problem is formulated and the proposed algorithm is described in \Cref{sec:algo} along with its proof of convergence under well-established conditions.
In \Cref{sec:NE}, numerical results are presented for various loss functions, parameter values and datasets, 
followed by concluding remarks in \Cref{sec:conclude}.

\section{Iterative Learning with Constraints} \label{sec:algo}

In this section, we formulate the problem of supervised learning for regression with constraints.
Then we present the proposed iterative scheme with convergence proof. 

\subsection{Problem Formulation}

Consider a regression problem where we should predict the ideal output $ y \in \Re^n $ given the inputs $ X \in \Re^{n \times d} $.
Here $ n $ denotes the number of points in the dataset, and $ d $ corresponds to the number of features in each data point.
Let the model for supervised learning be denoted by $ \hat{y} = f(X, \theta) $, where $\theta\in\Re^p$ is the model parameter and $ \hat{y} \in \Re^{n} $ is the output predicted by the current model parameter.
The goal of regression is to identify the optimal model parameter $\theta^*$ that minimizes a given loss function, $ L(y, \hat{y}) $, $ L: \Re^{n} \times \Re^{n} \to \Re $.
In addition, we enforce constraints on the predicted output so that it belongs to a feasible set denoted by $C\subset\Re^n$, i.e.,  $\hat{y} \in C $.
Thus, the optimization problem for regression with constraints can be formulated as
\[
    \theta^* = \argmin_\theta \braces{L(y, \hat{y}) \mid \hat{y} = f(X, \theta), \text{ and } \hat{y} \in C}.
\]

Alternatively, this can be reorganized into an optimization on the output space as
\begin{gather}
	z = \argmin_{\hat{y}} \braces{L(\hat{y}, y) \mid \hat{y} \in B \cap C},\label{eqn:y_star}\\
	\theta^* = \argmin_{\theta} \{ L(z, \hat y) \mid \hat y = f(X,\theta), \theta\in\Re^p\}.\label{eqn:theta_star}
\end{gather}
as proposed by~\cite{detassis2020teaching}, where $ B = \braces{\hat{y} \mid \hat{y} = f(X, \theta), \theta\in\Re^p} $ is the set of all possible outputs under the current model.
In other words, \eqref{eqn:y_star} is to find an alternative optimal target $z\in\Re^n$ that is closest to the ideal target $y$ under the restriction of the given constraint and the model bias.
Next, in \eqref{eqn:theta_star}, the model parameter is optimized such that the predicted output matches to the optimal target $z$, not the ideal target $y$.
The intriguing feature is that the supervised learning in \eqref{eqn:theta_star} corresponds to the usual supervised learning without constraints, as the constraints are enforced indirectly through \eqref{eqn:y_star}.
As such, any supervised learning scheme can be utilized for \eqref{eqn:theta_star}.
For \eqref{eqn:y_star}, standard tools in nonlinear constrained optimization can be applied. 

\subsection{Iterative Learning Algorithm with Constraints}

In \cite{detassis2020teaching}, this problem is tackled by a clever combination of two iterations, which is verified by various numerical examples.
But it might be heuristic in the sense that no convergence property is established. 
Here we propose the following alternative iterative scheme for \eqref{eqn:y_star} and \eqref{eqn:theta_star}, summarized by Algorithm \ref{alg:af}, which provides a certain convergence property in regression. 
Here, $ \alpha, \beta $ are non-negative parameters in the adjustment step, and $ N_i $ is the total number of iterations of this procedure.
\algsetup{
	indent=2em,
	linenodelimiter=.,
}
\begin{algorithm}[h!]
\linespread{1.25}\selectfont
	\begin{algorithmic}[1]
		\REQUIRE $ y \in \Re^n, \braces{\alpha, \beta} \in \Re, N_i \in \mathbb{Z} $
		\STATE $ \hat{y}^1 = \argmin_{\hat{y}} \braces{L(\hat{y}, y) \mid \hat{y} \in B} $ \hfill \COMMENT {Initial training}
		\FOR {$ i = 1$ \TO $N_{i}-1$}
		\IF {$ \hat{y}^i \notin C $} 
		\STATE $ z^i = \argmin_{z} \braces{L(z, (1-\alpha)y + \alpha \hat{y}^i) \mid z \in C} $ \hfill \COMMENT {Infeasible adjustment}
		\ELSE
		\STATE $ z^i = \argmin_{z} \braces{L(z, y) \mid L(z, \hat{y}^i) \le \beta,\ z \in C} $ \hfill \COMMENT {Feasible adjustment}
		\ENDIF
		\STATE $ \hat{y}^{i+1} = \argmin_{\hat{y}} \braces{L(\hat{y}, z^i) \mid \hat{y} \in B} $ \\ \COMMENT {Unconstrained training}
		\ENDFOR
		\ENSURE $\hat{y}^{N_i}$
	\end{algorithmic}
	\caption{Regression with constraints}
	\label{alg:af}
\end{algorithm}

In the first step of initial training, supervised learning is performed without considering the constraint. 
The next iterations are composed of two parts of target adjustment and unconstrained training, and the target adjustment step has two sub-cases depending on the output of the previous step. 
In particular, the most critical step is when the output of the trained model does not satisfy the constraint. 
In the step 4, denoted by \textit{infeasible adjustment}, the target is adjusted to minimize $ L(z, (1-\alpha)y + \alpha \hat{y}^i) $.
That is, we find a feasible target $ z \in C $ that is closest to a point on the line connecting $ y $ and $ \hat{y} $ in terms of the loss function. 
This is in opposition to obtaining a vector that considers the original label, $ y $ and the current prediction, $ \hat{y} $ separately, as presented by~\cite{detassis2020teaching} in the form of $L(z,y)+\frac{1}{\alpha}L(z,\hat y)$.

Next, when the output of the trained model satisfies the constraint, in the step of \textit{feasible adjustment}, the target $z$ is moved closer to the original target $y$ within a ball of radius $\beta$ measured in terms of the loss. 
Finally, the model is trained with the adjusted target, and the whole procedure is repeated. 

In the proposed algorithm, the key idea is selecting the objective function of the infeasible adjustment as $L(z, (1-\alpha)y + \alpha \hat{y}^i) $.
This establishes the convergence property presented in the next subsection, and it improves numerical properties as illustrated in \Cref{sec:NE}.  
Interestingly, for a specific choice of the loss function, namely mean squared error (MSE) loss, it is equivalent to the form of $L(z,y)+\frac{1}{\alpha}L(z,\hat y)$ as described below.

\begin{remark}\label{rem:alpha}
	If the loss function is mean squared error, the procedure in Algorithm \ref{alg:af} and the Moving Targets algorithm in \cite{detassis2020teaching} are equivalent after adjusting the parameter $\alpha$.
\end{remark}
\begin{proof}
	Since the main difference between the two is in the infeasible adjustment case, we compare the corresponding optimization problems.
    With $ L(z, y) = (1/n) \sum_{k=1}^{n} (z_k - y_k)^2 $ as the MSE loss, Algorithm \ref{alg:af} addresses
	\begin{align*}
		z_a &= \argmin_{z} \braces{\sum_{k=1}^{n} (z_k - (1-\alpha_a)y_k - \alpha_a \hat{y}_k)^2 \mid z \in C} \\
		&= \argmin_{z} \braces{\sum_{k=1}^{n} z_k^2  - 2 (1-\alpha_a) z_k y_k - 2 \alpha_a z_k \hat{y}_k \mid z \in C}.
	\end{align*}
	Whereas the master step from Moving Targets is,
	{\small \begin{align*}
		z_m &= \argmin_{z} \braces{\sum_{k=1}^{n} (z_k - y_k)^2 + \frac{1}{\alpha_m} \sum_{k=1}^{n} (z_k - \hat{y}_k)^2 \mid z \in C} \\
		=& \argmin_{z} \braces{\frac{1}{\alpha_m} \sum_{k=1}^{n} (\alpha_m +1)z_k^2 - 2 \alpha_m z_k y_k - 2 z_k \hat{y}_k \mid z \in C}.
	\end{align*}}
	Here, the subscript $ a $ represents our algorithm while variables with $ m $ as the subscript are from Moving Targets.
	The objective functions in $ z_a, z_m $ differ only by a scale if,
	\begin{equation}\label{eqn:alpha_comp}
		\alpha_a (\alpha_m + 1) = 1.
	\end{equation}
	Hence, the solutions that are obtained from them will be identical, i.e., $ z_a = z_m $.
\end{proof}

\subsection{Convergence Property}\label{sec:conv}

Now we present a convergence property of Algorithm \ref{alg:af}, which has motivated the proposed form of the objective function in the infeasible adjustment step. 

Let any norm on the Euclidean space be denoted by $ \norm{\cdot} : \Re^n \to \Re $.
Also, the Euclidean $L_2$ norm and the $L_1$ norm are denoted by by $ \norm{\cdot}_2$, and  $\norm{\cdot }_1 $, respectively.
A projection operator $ P_{Z,L} : \Re^n \to \Re^n $ on the set $ Z $ with respect to the loss $ L $ is defined as
\begin{equation}\label{eqn:proj}
	P_{Z,L}(x) = \argmin_{z} \braces{L(z, x) \mid z \in Z}.
\end{equation}
In other words, $x\in\Re^n$ is projected to $z\in Z$ such that the distance between $x$ and $z$ is minimized in terms of the loss $L$.

Consider a convex subset $ Z \subseteq X $ of a finite-dimensional normed vector space $ (X, \norm{\cdot}) $.
There exists a unique projection $ P_{Z,\norm{\cdot}}(x) \in Z $ for each $ x \in X $ such that
\[
\norm{x - P_{Z,\norm{\cdot }}(x)} = \inf \braces{\norm{x-z} \mid z \in Z}
\]
if the underlying geometric constraint is satisfied (see \cite[Proposition 3.2]{balestro2019convex}).
That is, $ P_{Z,\norm{\cdot}} $ should not be contained in some non-degenerate line segment of $ \partial Z $ which is parallel to some non-degenerate line segment in the boundary of the unit $ \norm{\cdot} $ ball.

\begin{assumption}
We make the following assumptions.
    \begin{itemize}
        \item The sets $B$ and $C$ are convex.
        \item The projection operator in $B$ and $C$ is Lipschitz, i.e., there exists a norm $\|\cdot\|$ and $K>0$ such that $ \norm{P_{A,L}(x) - P_{A,L}(y)} \le K \norm{x - y} $ for all $ x, y \in \Re^n $ where $A=B$ or $A=C$.
\end{itemize}
\end{assumption}
When the loss function in the projection~\eqref{eqn:proj} is MSE, the above two statements are actually equivalent~\cite{balestro2019convex}.

Now we are concerned with convergence of the sequence, $ (\hat{y}^i) \in B $ generated after the training step.
In other words, we wish to show that $ \hat{y}^i \to \bar{y} $ as $i\rightarrow\infty$ for some $\bar y \in\Re^n$.
The convergence of Algorithm \ref{alg:af} is established as follows.

\begin{theorem}\label{thm:conv}
	Suppose $ \alpha < 1 / K^2 $, where $ K \ge 1 $ is the Lipschitz constant introduced in Assumption 1.
    The iterations of Algorithm \ref{alg:af} has a unique fixed point in $ B $, which is the limit of the sequence $ (\hat{y}^i) $ for an initial $  \hat{y}^1 \in B $, when $\beta$ is sufficiently small. 
\end{theorem}
\begin{proof}
    When $\beta\rightarrow 0$, Algorithm \ref{alg:af} iterates between the infeasible adjustment step and unconstrained training, and it can be written as
\begin{itemize}
	\item Affine extension: $ y^\alpha = (1-\alpha)y + \alpha \hat{y}^i $
	\item Adjustment: $ z^i = P_{C, L}(y^\alpha) = \argmin_{z} \braces{L(z, (1-\alpha)y + \alpha \hat{y}^i) \mid z \in C} $
	\item Learning: $ \hat{y}^{i+1} = P_{B, L}(z^i) = \argmin_{\hat{y}} \braces{L(\hat{y}, z^i) \mid \hat{y} \in B} $
\end{itemize}
Therefore, Algorithm \ref{alg:af} corresponds to a concatenation of two projections as
\begin{equation}
	\hat{y}^{i+1} = P_{B,L}(P_{C,L}(h(\hat{y}^i))),
\end{equation}
where $ h:\Re^n\rightarrow\Re^n $ is the affine extension function defined as $ h(\hat{y}^i) = (1-\alpha)y + \alpha \hat{y}^i $.

Consider any two points $  \hat{y}^1,  \hat{y}^2 \in B $. 
We have
\begin{align*}
	&\norm{P_{B,L}(P_{C,L}(h( \hat{y}^1))) - P_{B,L}(P_{C,L}(h( \hat{y}^2)))}  \\
	&\le K \norm{P_{C,L}(h( \hat{y}^1)) - P_{C,L}(h( \hat{y}^2))} \\
    & \le K^2 \norm{h( \hat{y}^1) - h( \hat{y}^2)} \\
	& \le K^2 \alpha \norm{ \hat{y}^1 -  \hat{y}^2}
\end{align*}
If $ K^2 \alpha < 1 $, then each iteration is a contraction mapping on $ B $ with the metric induced by this norm, $ d(x,y) = \norm{y - x} $.
Since $ (B, d) $ is a complete metric space, the series of iterations has a unique fixed point $ \bar{y} = f(\bar{y}) $ according to the Banach fixed point theorem~\cite{ciesielski2007stefan}.
Moreover, the sequence $\{\hat y^1, \hat y^2,\ldots\}$ converges to $\bar y$ for any $  \hat{y}^1 \in B $.
\end{proof}

In short, the Lipschitz properties of Assumption 1 ensures that each iteration is a contraction.
The critical question is how we can ensure the Lipschitz property of the projection operators. 

\begin{corollary}\label{cor:conv}
    The convergence of Algorithm \ref{alg:af} is guaranteed as described in Theorem \ref{thm:conv} for the following loss functions. 
\begin{itemize}
	\item For the mean squared error (MSE) given by $ L(z, y) = \frac{1}{n} \sum_{k=1}^{n} (z_k - y_k)^2 $,
the algorithm converges for the parameter $ \alpha \in [0, 1) $.
	\item For the mean absolute error (MAE) given by $ L(z, y) = \frac{1}{n} \sum_{k=1}^{n} |z_k - y_k| $,
the algorithm converges for the parameter $ \alpha \in [0, 0.25) $.
\end{itemize}
\end{corollary}
\begin{proof}
For MSE, the projection in \eqref{eqn:proj} corresponds to
\begin{align*}
	P_{Z,L}(y) &= \argmin_{z} \braces{\frac{1}{n} \sum_{k=1}^{n} (z_k - y_k)^2 \,\middle\vert\, z \in Z} \\
	&= \argmin_{z} \braces{\norm{z - y}_2^2 \mid z \in Z},
\end{align*}
which is equal to minimization with respect to the standard Euclidean norm, $ \norm{\cdot }_2 $.
The proximity map for a closed convex set in the Hilbert space with Euclidean inner product satisfies the condition $ \norm{Px - Py} \le \norm{x - y} $~\cite{cheney1959proximity, balestro2019convex}.
Since its Lipschitz constant is $ K = 1 $, according to Theorem \ref{thm:conv}, Algorithm \ref{alg:af} converges for $ \alpha \in [0, 1) $.

Next, for the MAE loss, the projection becomes
\begin{align*}
	P_{Z,L}(y) &= \argmin_{z} \braces{\frac{1}{n} \sum_{k=1}^{n} |z_k - y_k| \,\middle\vert\, z \in Z} \\
	&= \argmin_{z} \braces{\norm{z - y}_1 \mid z \in Z},
\end{align*}
which is optimization with respect to the $ L_1 $ norm, $ \norm{\cdot }_1 $.
It is shown by \cite{de1967radial} that the Lipschitz constant is $ K = 2 $ with respect to the $L_1$ norm.
Hence convergence is guaranteed if $ \alpha \in [0, 0.25) $.
\end{proof}

\Cref{cor:conv} is the main result of this paper establishing the convergence of iterative algorithm for regression with constraints. 
Next, we show that the proposed algorithm further exhibits improved numerical properties in several examples, beyond providing mathematical assurance.

\section{Numerical Simulation}\label{sec:NE}

We evaluate the performance of the proposed algorithm with various datasets, parameter values, and loss functions.
First, we underscore that this section is meant to be an exercise in understanding an algorithmic procedure, and the resulting output is supposed to be interpreted as purely technical results. 
The type of constraints that we are going to consider for regression is called fairness constraints in socially sensitive decision making (see \cite{aghaei2019learning}), which is measured in the form of Disparate Impact Discrimination Index as
\begin{equation}\label{eqn:DIDI}
DIDI^r(z) = \sum_{p \in P}\sum_{v \in D_p} \left| \frac{1}{n} \sum_{i=1}^{n} z_i - \frac{1}{|X_{p,v}|} \sum_{i \in X_{p,v}} z_i \right| \le \epsilon.
\end{equation}
Here $ D_p $ is the set of values for the $ p $-th protected feature, such as gender or disability, from the set $ P $, and $ X_{p,v} $ represents the inputs whose $ p $-th feature has value $ v $.
Roughly speaking, it represents the difference between the mean of the output and the mean conditioned by the protected feature, and 
the higher DIDI, the more the dataset suffers from disparate impact. 
The constraint on the DIDI value, $ \epsilon $, is taken to be a fraction ($ 0.2 $) of the DIDI value for the training set.

Three different datasets are considered for this regression problem with fairness constraints:
\begin{itemize}
	\item \texttt{student} dataset ($ n = 649 $ points, $ d = 33 $ attributes) for Portuguese class from the UCI repository which has been used to predict secondary school student performance in \cite{cortez2008using}.
	We are going to protect the feature, \textit{sex}, while trying to estimate the final grade of each student, \textit{G3}.
	Meanwhile, features like \textit{romantic} interests which will likely have no relation to the output are removed according to \cite{cortez2008using}.
	\item \texttt{crime} dataset also from the UC Irvine Machine Learning repository~\cite{Dua2019} which has $ n = 2,215,\ d = 147 $.
	Since the target variable is \textit{violentPerPop} representing per capita violent crimes, we want to impose fairness constraints w.r.t. the protected feature \textit{race}.
	Features that have a lot of NaN values are removed along with others, which are directly dependent on the targets and act as outputs themselves.
	\item \texttt{blackfriday} dataset which is available online at \cite{bfdata}.
	The original training data can not be utilized with the limited amount of computing resources available since it is very large ($ n \approx 550,000,\ d = 12 $).
	So, we select a sample of data from the start with size, $ n = 50,000 $.
	Here, the goal is to estimate the amount of money spent, \textit{Purchase}, while ensuring that the predictions are fair with respect to the protected feature, \textit{Gender}.
	A new attribute, \textit{Product\_ID\_Count}, which is the value count of \textit{Product\_ID} is introduced since it represents the number of times a product has been purchased.
	Also, the identity features, \textit{User\_ID, Product\_ID} are removed.
\end{itemize}
All the categorical features in the data are encoded into an integer array using an Ordinal Encoder.
Finally, obtained values are normalized to be between $ 0 $ and $ 1 $ to ensure balanced regression.

\newcolumntype{m}{>$l<$} 
\begin{table}[h!]
	\caption{Parameters}
	\label{tab:alpha}
	\centering
	\begin{tabular}{mmm}
		\toprule
		& \multicolumn{2}{c}{$ \alpha $}\\
		\cmidrule(l){2-3}
		\text{Weight on } \hat{y} \text{ w.r.t } y  & \text{Algorithm \ref{alg:af}, } \alpha_a  & \text{Moving Targets~\cite{detassis2020teaching}, } \alpha_m \\
		\midrule
		\text{Less} & 0.1  & 9 \\
		\text{Equal}  & 0.5  & 1  \\
		\text{More} & 0.9  & 1/9  \\
		\bottomrule
	\end{tabular}
\end{table}
Next the values of parameters $ \alpha, \beta $ are chosen as follows.
Extensive tuning of these terms has already been completed by \cite{detassis2020teaching}, where it is observed that $ \beta = 0.1 $ works well empirically.
This value of $\beta=0.1$ is adopted here as well. 
For the parameter $\alpha$, three different values are chosen, and the corresponding values of the Moving Targets algorithm ~\cite{detassis2020teaching} are calculated from \eqref{eqn:alpha_comp} in Remark \ref{rem:alpha}.
These are listed at Table \ref{tab:alpha} according to the relative weight on $\hat y$ with respect to $y$ in the infeasible adjustment step. 
As $\alpha_a$ is increased, more weight is assigned to $\hat y$ that has been adjusted for the constraint, compared with the original target $y$.
Therefore, there is more emphasis on the satisfaction of the constraint.

For the machine learning model of regression, a gradient boosted tree is chosen as it ensures repeatability.
It also achieves higher accuracy as well as better constraint satisfaction. 
To study the convergence property, the algorithms are executed for the total of $ N_i = 30 $ iterations.
A five-fold cross validation is performed to obtain a reliable estimate of performance as well as the standard deviation.
We utilize a computer with a quad core Intel i7 CPU and Nvidia GK107 GPU with 16 GB RAM.
To solve the optimization problems in the adjustment step of Algorithm \ref{alg:af}, we utilize the IBM software CPLEX~\cite{cplex2009v12} for MSE and MAE.
Additionally, we consider the mean Huber loss (MHL) 
$  L(z, y) = \frac{1}{n} \sum_{k=1}^{n} g(z_k - y_k) $, where 
\begin{equation}\label{eqn:huber}
	g(x) = \begin{cases}
		x^2, & |x| \le M \\
		2M|x| - M^2, & |x| > M
	\end{cases}
\end{equation}
with $M=0.1$, which is implemented by CVXPY~\cite{diamond2016cvxpy}.

\subsection{Results}
\newcolumntype{m}{>$c<$} 
\begin{table*}[h!]
	\caption{Performance after $ N_i = 30 $ iterations; shown as \textit{mean (std)} of 5 folds}
	\label{tab:results}
	\centering
	{\small
	\begin{tabular}{mmm|mmmmmm}
		\toprule
		& & & \multicolumn{2}{c}{\texttt{crime}} & \multicolumn{2}{c}{\texttt{student}} & \multicolumn{2}{c}{\texttt{blackfriday}}\\
		\cmidrule(lr){4-5} \cmidrule(lr){6-7} \cmidrule(lr){8-9}
		\text{Loss} & \alpha_a & & \text{A} & \text{M} & \text{A} & \text{M} & \text{A} & \text{M} \\
		\midrule
		
		\multirow{6}{*}{MSE}
		& 0.1 & R^2 	 & \multicolumn{2}{m}{.550\ (.013)} & \multicolumn{2}{m}{.921\ (.010)} & \multicolumn{2}{m}{.645\ (.002)} \\
		&  & \mathcal{C} & \multicolumn{2}{m}{.262\ (.013)} & \multicolumn{2}{m}{.325\ (.044)} & \multicolumn{2}{m}{.567\ (.016)} \\
		& 0.5 & R^2 	 & \multicolumn{2}{m}{.494\ (.012)} & \multicolumn{2}{m}{.908\ (.012)} & \multicolumn{2}{m}{.620\ (.001)} \\
		&  & \mathcal{C} & \multicolumn{2}{m}{.237\ (.006)} & \multicolumn{2}{m}{.289\ (.027)} & \multicolumn{2}{m}{.511\ (.026)} \\
		& 0.9 & R^2 	 & \multicolumn{2}{m}{.368\ (.004)} & \multicolumn{2}{m}{.881\ (.022)} & \multicolumn{2}{m}{.481\ (.003)} \\
		&  & \mathcal{C} & \multicolumn{2}{m}{.216\ (.004)} & \multicolumn{2}{m}{.241\ (.006)} & \multicolumn{2}{m}{.358\ (.013)} \\
		\midrule
		
		\multirow{6}{*}{MAE}
		& 0.1 & R^2 	 & .520\ (.016) & .534\ (.014) & .891\ (.018) & .888\ (.027) & .645\ (.003) & .647\ (.002) \\
		&  & \mathcal{C} & .260\ (.014) & .280\ (.007) & .331\ (.054) & .318\ (.041) & .582\ (.018) & .577\ (.017) \\
		& 0.5 & R^2 	 & \mathbf{.467}\ (.019) & \mathbf{.342}\ (.085) & .874\ (.019) & .883\ (.029) & \mathbf{.624}\ (.002) & \mathbf{.590}\ (.003)\\
		&  & \mathcal{C} & \mathbf{.239}\ (.013) & \mathbf{.265}\ (.005) & .333\ (.054) & .327\ (.026) & \mathbf{.478}\ (.018) & \mathbf{.577}\ (.028) \\
		& 0.9 & R^2 	 & .383\ (.062) & .359\ (.043) & .799\ (.051) & .785\ (.047) & \mathbf{.502}\ (.002) & \mathbf{.295}\ (.005)\\
		&  & \mathcal{C} & .220\ (.011) & .215\ (.007) & \mathit{.278}\ (.026) & \mathit{.212}\ (.021) & \mathit{.256}\ (.010) & \mathit{.219}\ (.006)\\
		\midrule
		
		\multirow{6}{*}{MHL}
		& 0.1 & R^2 	 & .530\ (.013) & .534\ (.013) & .921\ (.011) & .923\ (.010) & \multicolumn{2}{c}{------} \\
		&  & \mathcal{C} & .272\ (.008) & .276\ (.012) & .326\ (.048) & .311\ (.049) & \multicolumn{2}{c}{------} \\
		& 0.5 & R^2 	 & .493\ (.010) & .489\ (.013) & .907\ (.013) & .900\ (.015) & \mathbf{.620}\ (.001) & \mathbf{.611}\ (.003) \\
		&  & \mathcal{C} & .248\ (.007) & .258\ (.006) & .289\ (.035) & .291\ (.044) & .511\ (.026) & .509\ (.025)\\
		& 0.9 & R^2 	 & \mathbf{.368}\ (.004) & \mathbf{.318}\ (.009) & .882\ (.022) & .860\ (.024) & \multicolumn{2}{c}{------} \\
		&  & \mathcal{C} & \mathit{.217}\ (.002) & \mathit{.207}\ (.004) & .241\ (.006) & .232\ (.010) & \multicolumn{2}{c}{------} \\
		\bottomrule
	\end{tabular}
	}
\end{table*}

\Cref{tab:results} presents the results for varying loss functions, datasets and $ \alpha $ values.
Performance is measured through the regression coefficient $ R^2 $, and the ratio $\mathcal{C}$ of DIDI~\eqref{eqn:DIDI} of the predicted output to training data.
We also compare between our algorithm (denoted by A) and the Moving Targets~\cite{detassis2020teaching} (M) for the corresponding $ \alpha $ values from Table \ref{tab:alpha}.
According to Remark \ref{rem:alpha}, both methods are equivalent for MSE.
For the \texttt{blackfriday} dataset, two cases of $\alpha$ are left out since they could not be solved with the available computing resources.
As discussed above, $\alpha_a$ represents the trade-off between the satisfaction of the constraint and the regression.
This is well reflected in Table \ref{tab:results}: as $\alpha_a$ is increased, $\mathcal{C}$ decreases at the cost of reduced $R^2$.

Next, the \textbf{bold} fonts in \Cref{tab:results} represent the cases for which our algorithm performs better than Moving Targets in a statistically meaningful manner, and the \textit{italic} fonts represent the opposite case.
The statistical importance is assumed to occur when $ |\mu_a - \mu_m| \ge \sigma_a + \sigma_m $, i.e., the difference between the mean figures is greater than the sum of their standard deviations.
It can be observed that our procedure performs better in terms of both $ R^2 $ and $ \mathcal{C} $ in more cases for both \texttt{crime} and \texttt{blackfriday} datasets.
For the \texttt{student} dataset, which is the smallest one ($ n = 649 $), the results are mostly comparable. 


\begin{figure*}[h!]
	\centerline{
		\subfigure[$ R^2 $ for $ \alpha_a = 0.1 $, using MAE]{
			\includegraphics[width=0.33\linewidth]{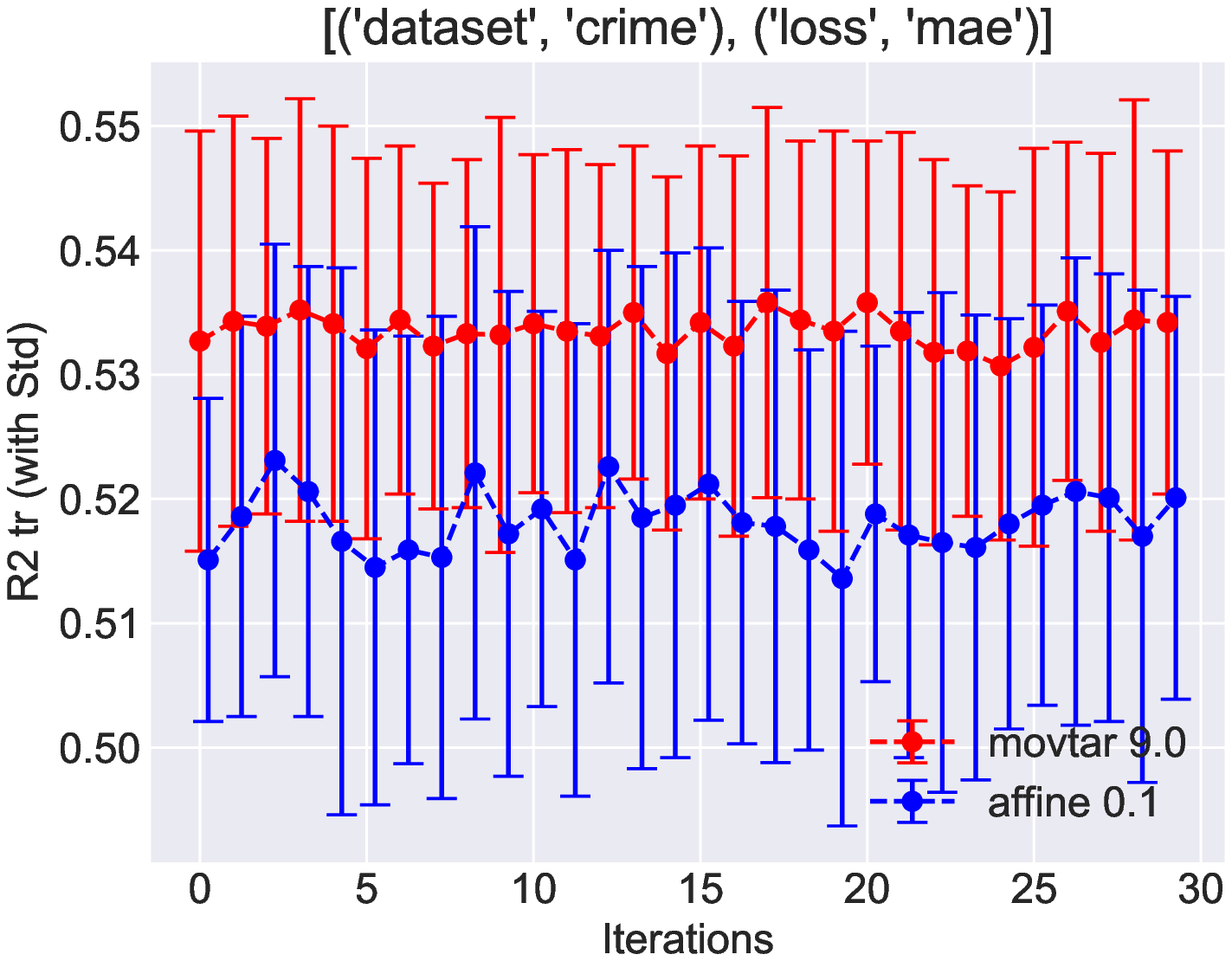}
		}
		\hfill
		\subfigure[$ R^2 $ for $ \alpha_a = 0.5 $, using MAE]{
			\includegraphics[width=0.33\linewidth]{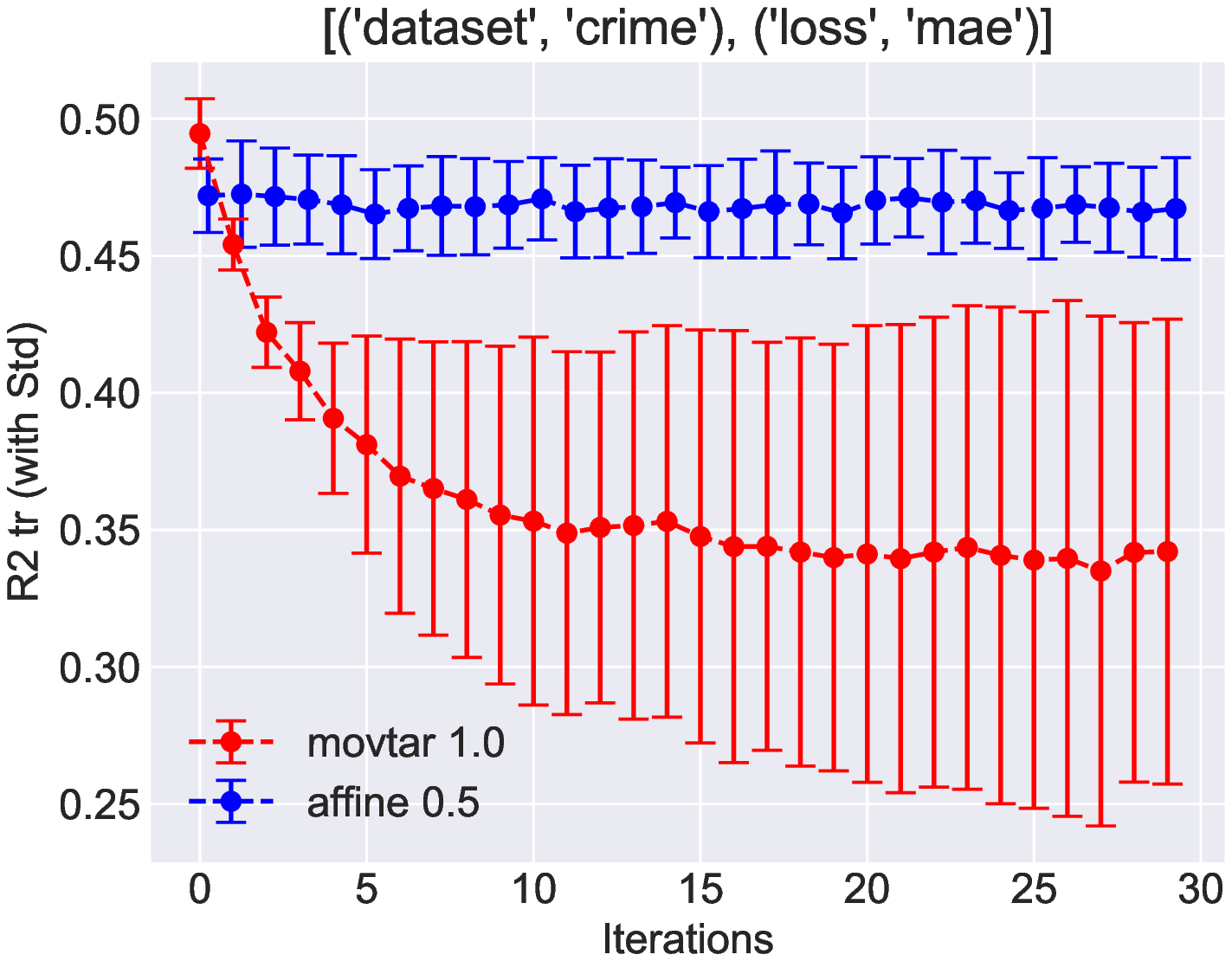}
			\label{fig:2}
		}
		\hfill
		\subfigure[$ R^2 $ for $ \alpha_a = 0.9 $, using MHL]{
			\includegraphics[width=0.33\linewidth]{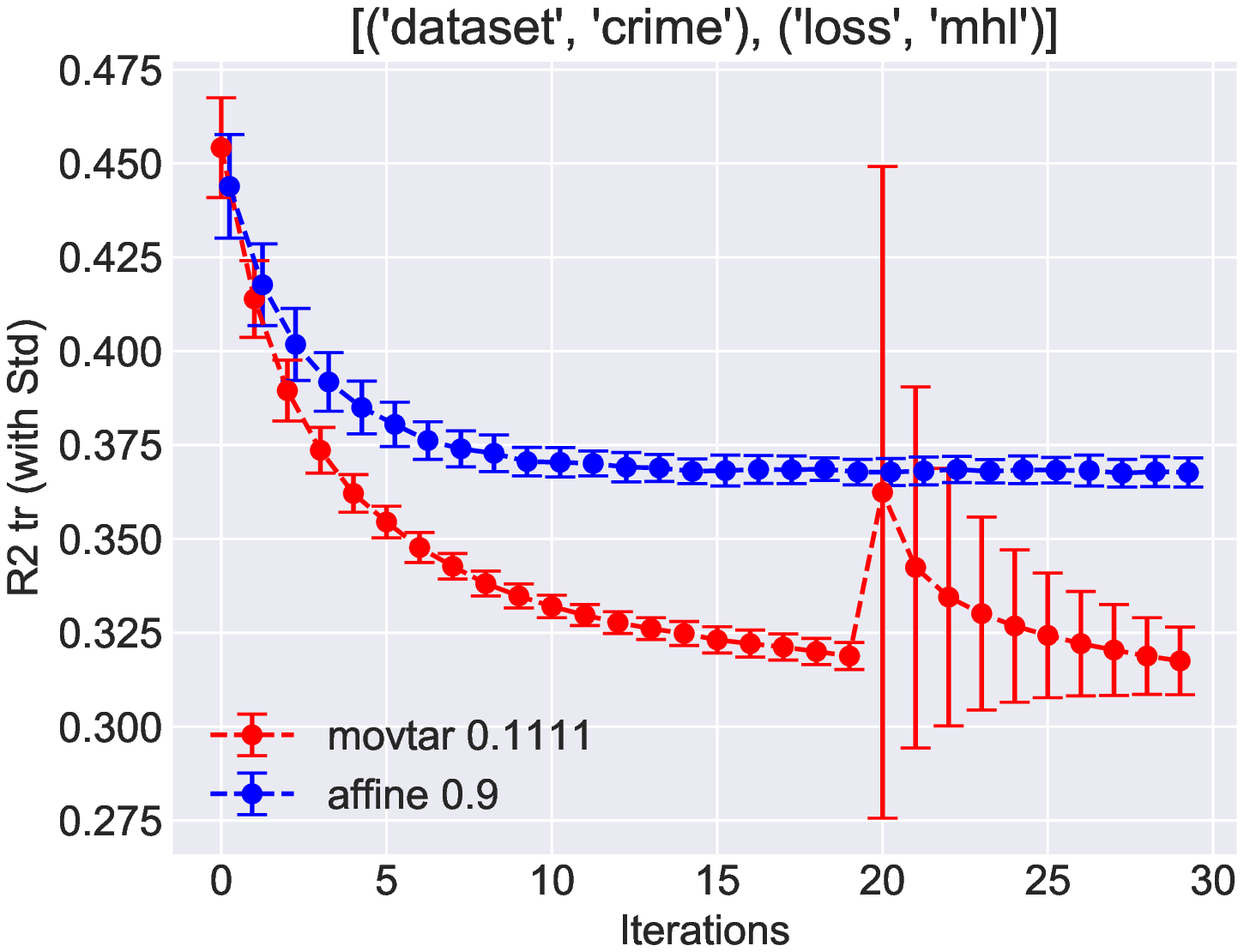}
		}
	}
	\centerline{
		\subfigure[$ \mathcal{C} $ for $ \alpha_a = 0.1 $, using MAE]{
			\includegraphics[width=0.33\linewidth]{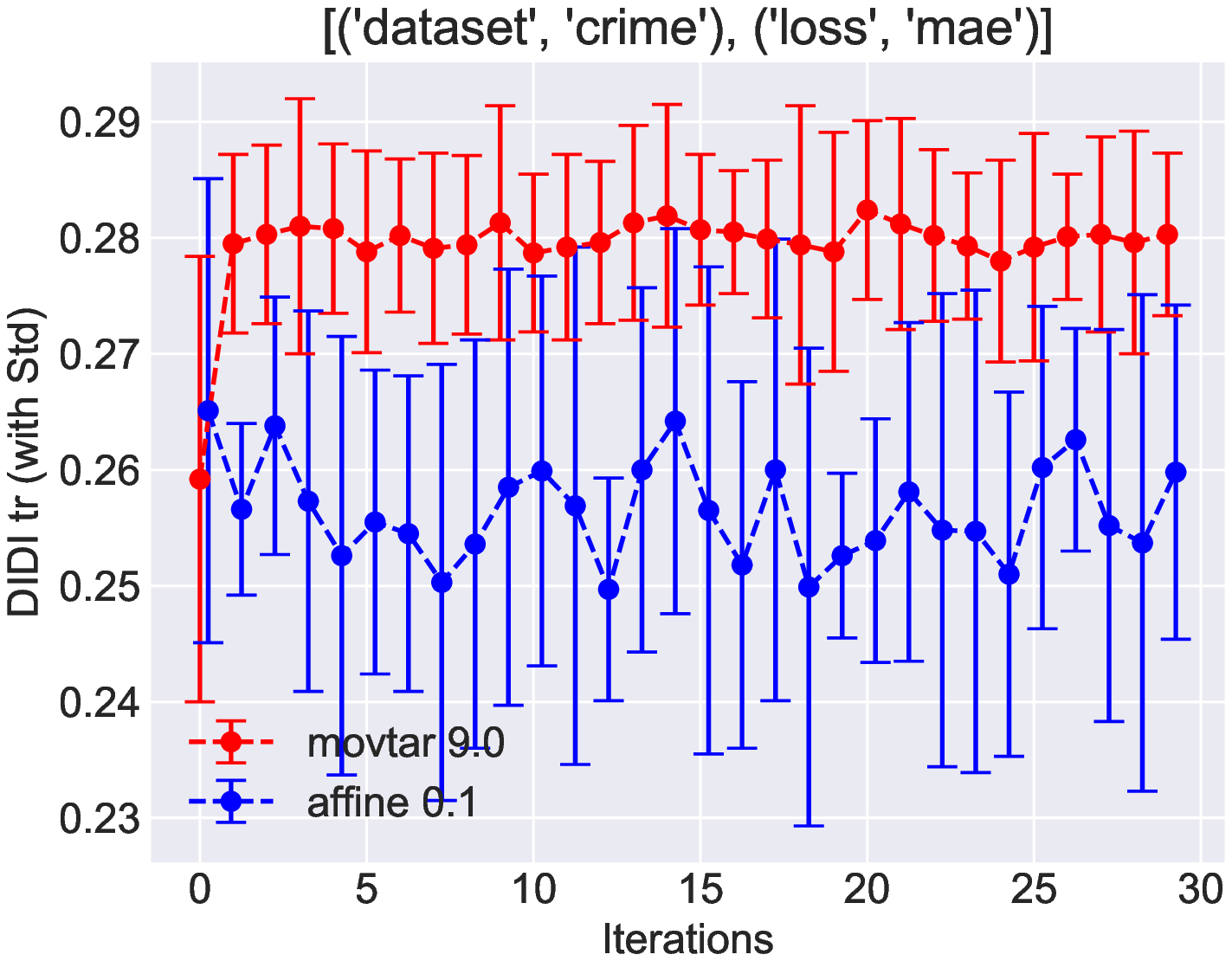}
		}
		\hfill
		\subfigure[$ \mathcal{C} $ for $ \alpha_a = 0.5 $, using MAE]{
			\includegraphics[width=0.33\linewidth]{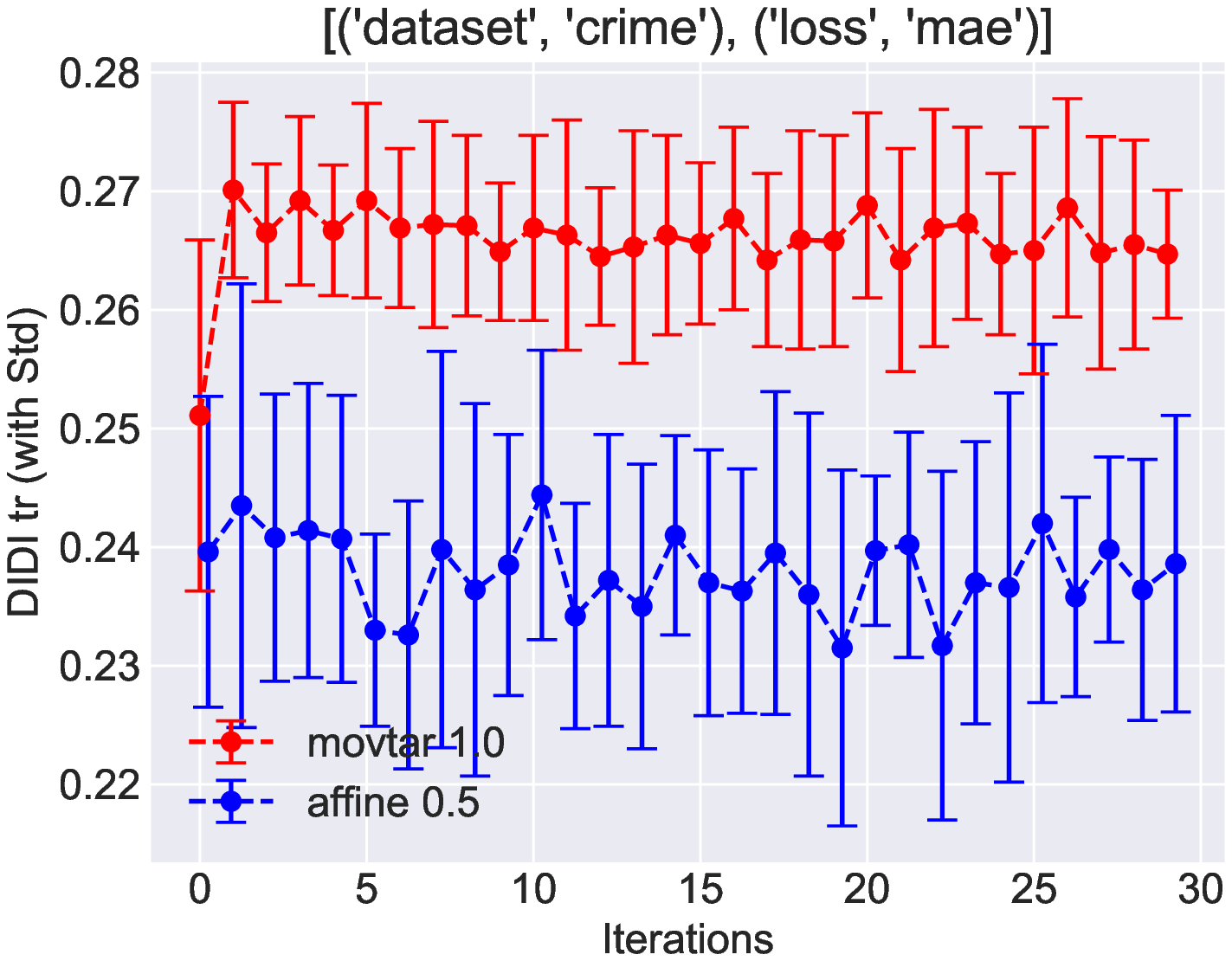}
		}
		\hfill
		\subfigure[$ \mathcal{C} $ for $ \alpha_a = 0.9 $, using MHL]{
			\includegraphics[width=0.33\linewidth]{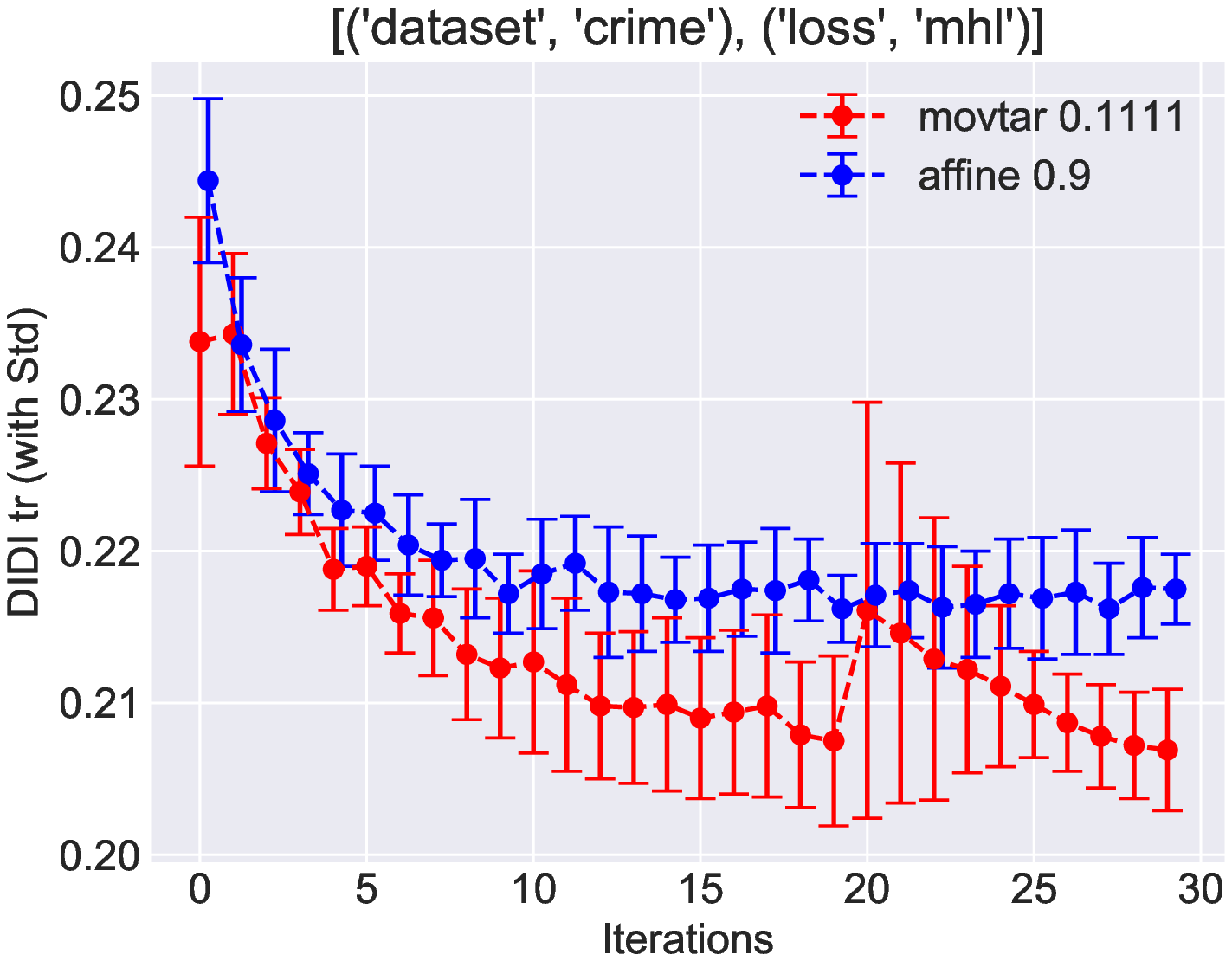}
		}
	}
	\caption{Comparison of our algorithm (blue) vs Moving Targets (red) for \texttt{crime} dataset; error bars represent standard deviation}
	\label{fig:crime}
\end{figure*}

Beyond the regression results summarized by \Cref{tab:results}, the advantages of the proposed approach are well illustrated by investigating the learning process.
Figures \ref{fig:crime} and \ref{fig:blackfriday} presents the evolution of $R^2$ and $\mathcal{C}$ over iterations for \texttt{crime} and \texttt{blackfriday} data, respectively.
When $ \alpha_a $ is small ($ 0.1 $), both the algorithms yield very similar results as seen in Figures \ref{fig:crime}.(a) and (d).

However, once $ \alpha_a $ is increased to $ 0.5 $ and $ 0.9 $ for more emphasis on constraint satisfaction, the proposed Algorithm \ref{alg:af} performs noticeably better.
As shown in Figures \ref{fig:crime}.(b) and (e), and also in Figures \ref{fig:blackfriday}.(a) and (c), the proposed approach yields a greater $R^2$ with a lower $\mathcal{C}$.
Next, in Figures \ref{fig:crime}.(c) and (f), and in Figures \ref{fig:blackfriday}.(b) and (d), it exhibits greater values of $R^2$ while being comparable in terms of the constraint satisfaction.
More importantly, the proposed approach displays more uniform performances over five-fold validation as the standard deviation is much lower, for example as illustrated by Figures \ref{fig:crime}.(b), \ref{fig:crime}.(c), and \ref{fig:blackfriday}.(c).
This suggests that the presented proof of convergence is in fact beneficial in numerical implementations as well, and the improved numerical properties in iterations may be more important for the scalability of regression and the complexity of constraints. 


\begin{figure*}[h!]
	\centerline{
		\subfigure[$ R^2 $ for $ \alpha_a = 0.5 $]{
			\includegraphics[width=0.33\linewidth]{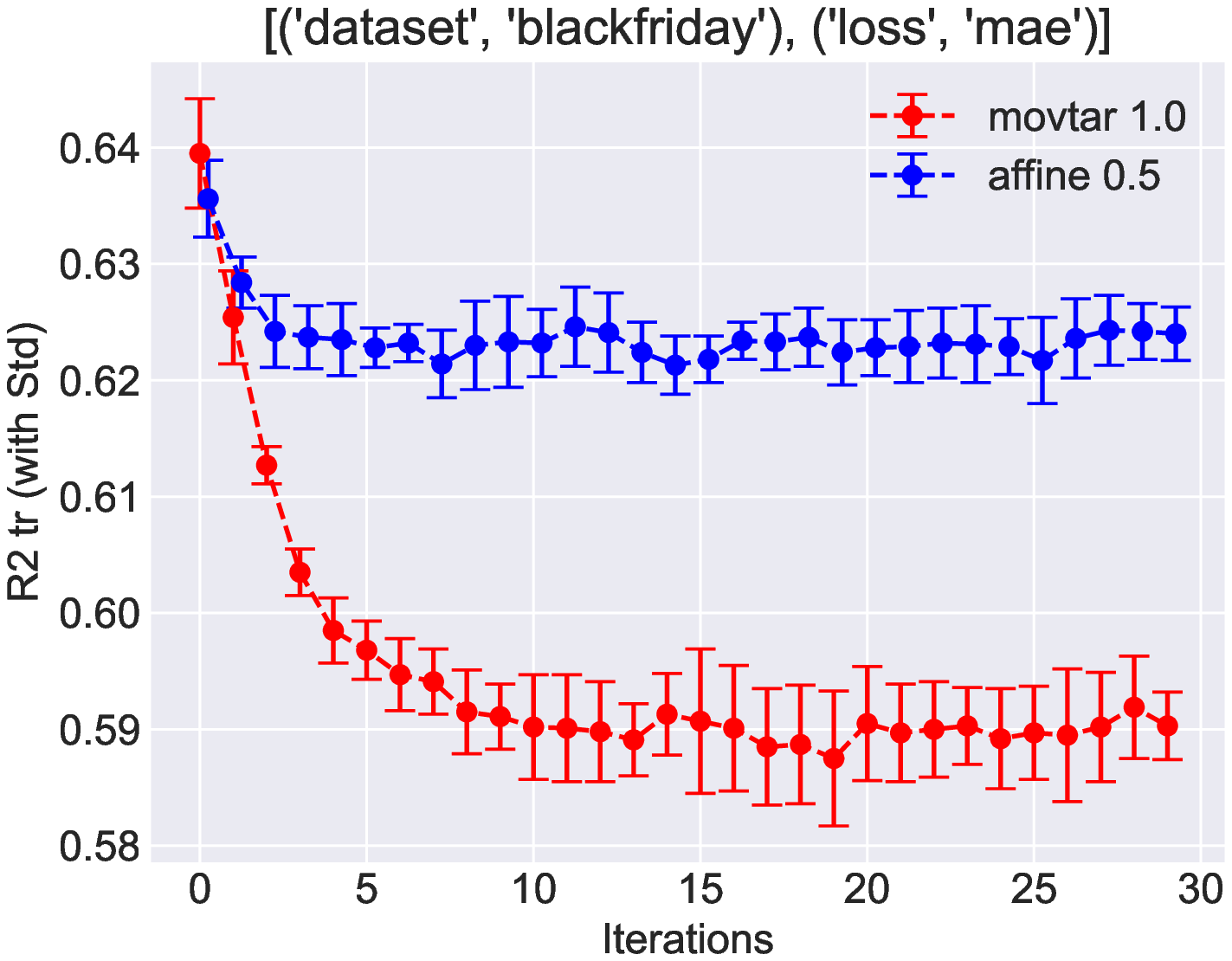}
		}
		\subfigure[$ R^2 $ for $ \alpha_a = 0.9 $]{
			\includegraphics[width=0.33\linewidth]{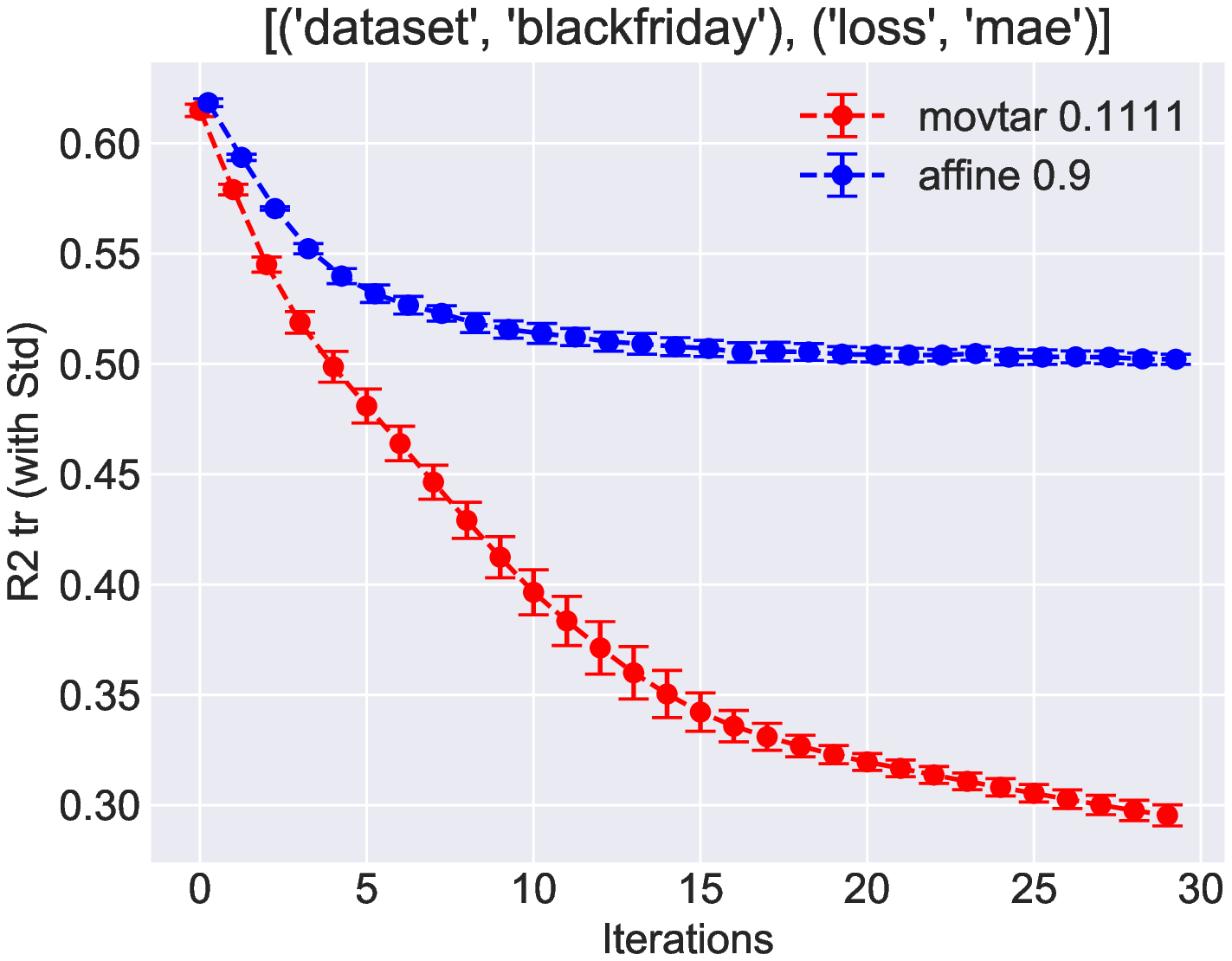}
		}
	}
	\centerline{
		\subfigure[$ \mathcal{C} $ for $ \alpha_a = 0.5 $]{
			\includegraphics[width=0.33\linewidth]{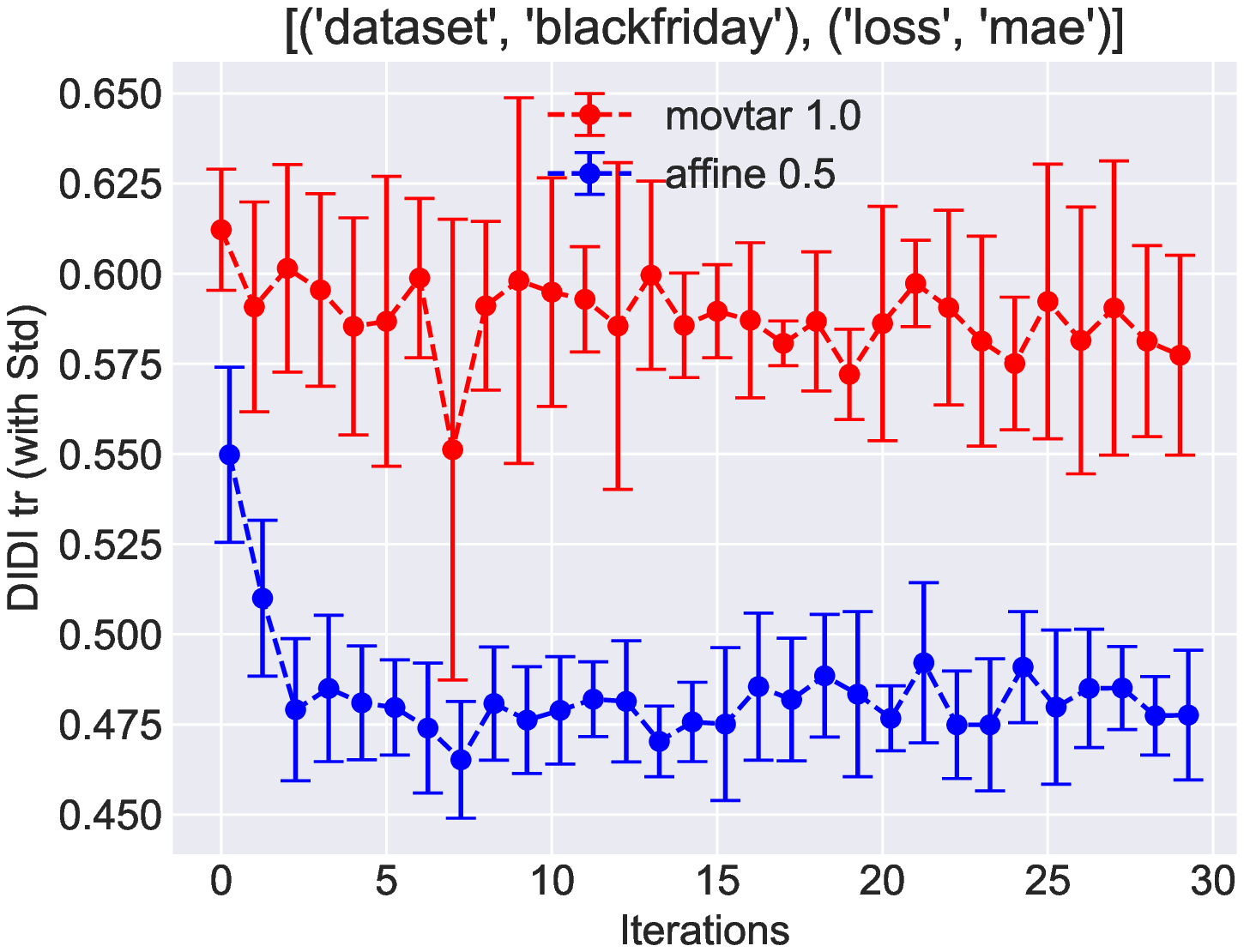}
		}
		\subfigure[$ \mathcal{C} $ for $ \alpha_a = 0.9 $]{
			\includegraphics[width=0.33\linewidth]{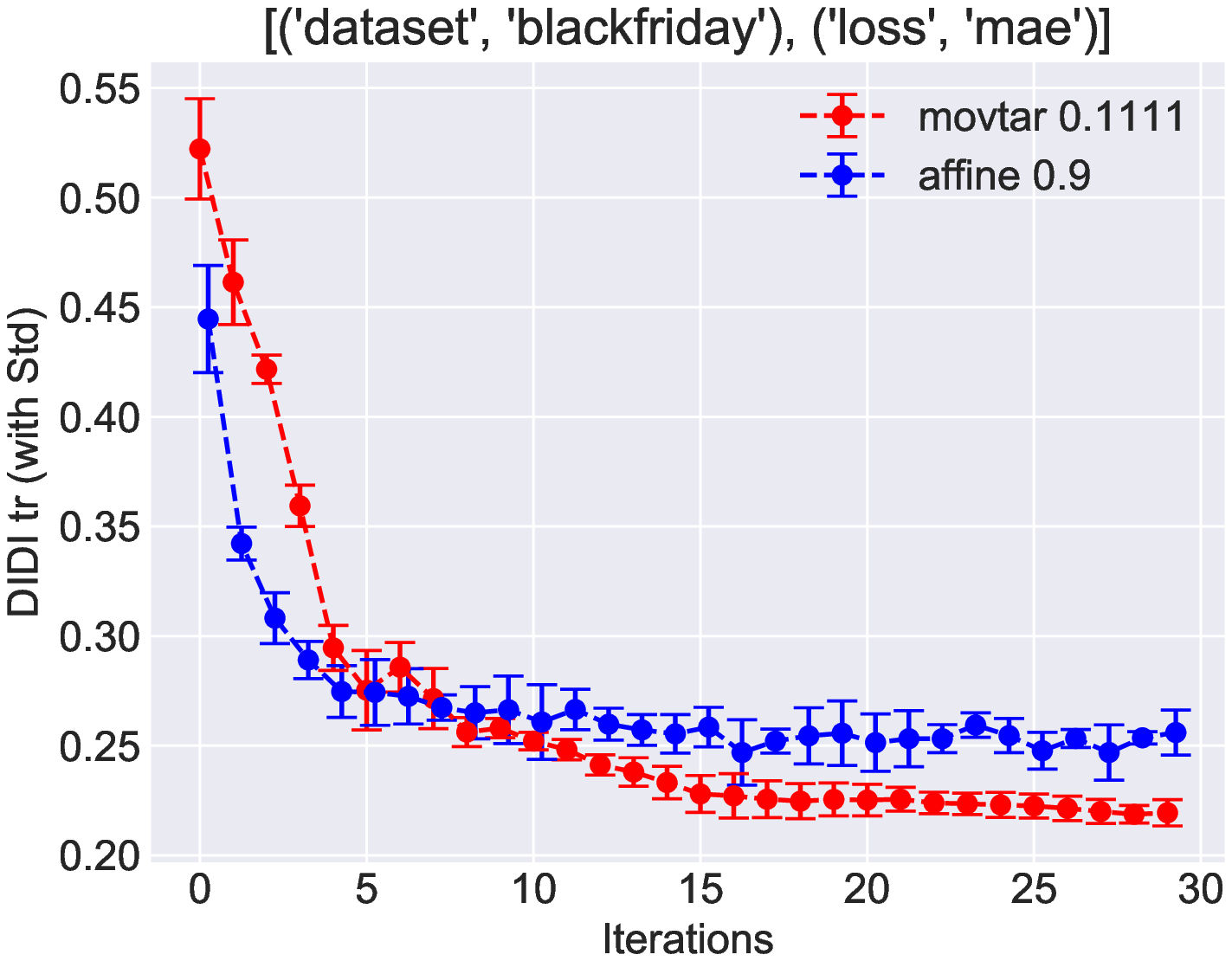}
		}
	}
	\caption{Comparison of our algorithm (blue) vs Moving Targets (red) for \texttt{blackfriday} dataset using MAE; error bars represent standard deviation}
	\label{fig:blackfriday}
\end{figure*}

\section{Conclusions}\label{sec:conclude}

We have proposed an iterative algorithm for regression  with constraints, composed of feasible/infeasible adjustments and training.
A convergence guarantee is also provided with an affine extension function in the infeasible adjustment step.
Furthermore, this result is specialized in the form of parameter constraints for selected loss functions.
Later, the results of numerical experiments are presented with varying datasets and parameter values.
The proposed convergence proof in supervised learning with constraints is the unique contribution, 
and it is further shown that the performances in all of the aspects of regression, constraint satisfaction and training stability are improved over the existing techniques. 
For future direction, we aim to study a convergence guarantee in more generic, non-Lipschitz conditions, and even for classification setups.

\bibliographystyle{ieeetran}
\bibliography{root.bib}

\end{document}